\documentclass[12pt]{amsart}
\usepackage{amsmath}
\usepackage{amstext}
\usepackage{amsfonts}
\usepackage{amssymb}
\usepackage{amsthm}
\usepackage{amsrefs}

\usepackage{microtype}
\usepackage{color,hyperref}
\definecolor{darkblue}{rgb}{0.0,0.0,0.3}
\hypersetup{colorlinks,breaklinks,
            linkcolor=red,urlcolor=red,
            anchorcolor=red,citecolor=red}

\usepackage{tikz}
\usetikzlibrary{arrows}
\usetikzlibrary{positioning}

\theoremstyle{plain}
\newtheorem{thm}{Theorem}[section]

\newtheorem{cor}[thm]{Corollary}
\newtheorem{prop}[thm]{Proposition}

\theoremstyle{definition}
\newtheorem{defn}[thm]{Definition}
\newtheorem{example}[thm]{Example}
\newtheorem{rem}[thm]{Remark}

\numberwithin{equation}{section}

\newcommand{\bR}{{\mathbb{R}}}

\renewcommand{\S}{{\mathcal{S}}}

\newcommand{\ex}{\operatorname{ex}}

\newcommand{\ca}{\mathrm{C}^*}

\newcommand{\fb}{\partial_F G}

\newcommand{\aut}{\operatorname{Aut}}

\newcommand{\cl}{\operatorname{cl}}

\begin{document}

%%%%%%%%%%%%%%%%%%%%%%%%%%%%%%%%%%%%%%%%%%
\title{An intrinsic characterization of C*-simplicity}

\author[M. Kennedy]{Matthew Kennedy}
\address{Department of Pure Mathematics\\ University of Waterloo\\
Waterloo, Ontario \; N2L 3G1 \\Canada}
\email{matt.kennedy@uwaterloo.ca}

\begin{abstract}
A group is said to be C*-simple if its reduced C*-algebra is simple. We establish an intrinsic (group-theoretic) characterization of groups with this property. Specifically, we prove that a discrete group is C*-simple if and only if it has no non-trivial amenable uniformly recurrent subgroups. We further prove that a group is C*-simple if and only if it satisfies an averaging property considered by Powers.
\end{abstract}

\subjclass[2010]{Primary 46L35; Secondary 20F65, 37A20, 43A07}
\keywords{discrete group, reduced C*-algebra, C*-simplicity}
\thanks{Author supported by NSERC Grant Number 418585.}

\maketitle

%%%%%%%%%%%%%%%%%%%%%%%%%%%%%%%%%%%%%%%%%%

\section{Introduction}
A group is said to be C*-simple if its reduced C*-algebra is simple, meaning that it has no non-trivial proper closed two-sided ideals. It has been an open problem for some time to determine an intrinsic group-theoretic characterization of groups that are C*-simple, along the lines of Murray and von Neumann's characterization of groups that give rise to factorial von Neumann algebras as groups with the infinite conjugacy class property.

It is not difficult to see that a C*-simple group necessarily has no non-trivial normal amenable subgroups, and based on a great deal of accumulated evidence, it was thought that this condition might be sufficient. However, in a major breakthrough, Le Boudec \cite{L2015} constructed examples showing that this is not the case.

The main result in this paper is an intrinsic group-theoretic characterization of C*-simplicity in terms of the nonexistence of non-trivial amenable subgroups satisfying a condition that is weaker than normality. We say that a subgroup $H$ of a discrete group $G$ is {\em residually normal} if there exists a finite subset $F \subset G \setminus \{e\}$ such that $F \cap gHg^{-1} \ne \emptyset$ for all $g \in G$.

\begin{thm} \label{thm:main}
A discrete group is C*-simple if and only if it has no amenable residually normal subgroups.
\end{thm}

We prove this result by analyzing the dynamics of the conjugation action of a group on its space of subgroups and invoking the dynamical characterization of C*-simplicity from \cite{KK2014}. The study of this action underlies the theory of uniformly recurrent subgroups introduced by Glasner and Weiss \cite{GW2015}. Using their terminology, Theorem \ref{thm:main} is equivalent to the following result (see Section \ref{sec:uniformly-recurrent-subgroups}). 

\begin{thm} \label{thm:c-star-simple-iff-no-urs-beginning}
A discrete group is C*-simple if and only if it has no non-trivial amenable uniformly recurrent subgroups.
\end{thm}

The notion of a uniformly recurrent subgroup can be viewed as a topological analogue of the measure-theoretic notion of an invariant random subgroup introduced by Ab{\'e}rt, Glasner and Vir{\'a}g \cite{AGV2014}. Many rigidity results in ergodic theory, and in particular results about the rigidity of characters on groups, can be viewed as results about the non-existence of certain invariant random subgroups. From this perspective, Theorem \ref{thm:c-star-simple-iff-no-urs-beginning} can be viewed as a kind of rigidity phenomenon in topological dynamics.

The theory of C*-simplicity began with Powers' theorem \cite{P1975} that free groups on two or more generators are C*-simple. The key insight underlying Powers' proof is that the left regular representation of these groups satisfies a very strong averaging property.

\begin{defn} \label{defn:powers-averaging-property-beginning}
A discrete group $G$ is said to have {\em Powers' averaging property} if for every element $a$ in the reduced C*-algebra $\ca_\lambda(G)$ and every $\epsilon > 0$ there are $g_1,\ldots,g_n \in G$ such that
\[
\left\| \frac{1}{n} \sum_{i=1}^n \lambda_{g_i} a \lambda_{g_i}^{-1} - \tau_\lambda(a) 1 \right\| < \epsilon,
\]
where $\tau_\lambda$ denotes the canonical tracial state on $\ca_\lambda(G)$.
\end{defn}

It is straightforward to show that any group satisfying Powers' averaging property is C*-simple. In fact, prior to the publication of \cite{KK2014} and \cite{BKKO2014}, essentially the only method for establishing the C*-simplicity of a given group was to show, often with great difficulty, that the group satisfied some variant of Powers' averaging property.

The next (perhaps somewhat surprising) result demonstrates the remarkable depth of Powers' insight. It turns out that every C*-simple group necessarily satisfies Powers' averaging property.

\begin{thm} \label{thm:intro-powers-averaging-property}
A discrete group is C*-simple if and only if it satisfies Powers' averaging property.
\end{thm}

In addition to this introduction there are five other sections. In Section~\ref{sec:preliminaries} we briefly review preliminary material. In Section~\ref{sec:boundary-maps-boundary-states} we clarify the relationship between C*-simplicity and the unique trace property, and obtain some technical results about the dual space of the reduced C*-algebra of a discrete group. In Section~\ref{sec:uniformly-recurrent-subgroups} we prove the characterization of C*-simplicity in terms of uniformly recurrent subgroups. In Section~\ref{sec:residually-normal-subgroups} we prove the main result characterizing C*-simplicity in terms of amenable residually normal subgroups. Finally, in Section~\ref{sec:powers-averaging-property} we prove that a group is C*-simple if and only if it has Powers' averaging property.

\subsection*{New developments} Since the first draft of this paper appeared in September 2015, a number of related developments have occurred. First, Haagerup \cite{H2015} independently obtained Theorem \ref{thm:intro-powers-averaging-property}.

Second, Theorem \ref{thm:c-star-simple-iff-no-urs-beginning} has been applied by Le Boudec and Matte Bon to study the C*-simplicity of groups of homeomorphisms of the circle, and in particular Thompson's groups $F$, $T$ and $V$. They proved that Thompson's group $V$ is C*-simple, and proved that the non-amenability of Thompson's group $F$ is equivalent to the C*-simplicity of $T$.

Third, Bryder and the author \cite{BK2016} applied similar ideas to study the maximal ideals of reduced (twisted) crossed products over C*-simple groups. In particular, we established a bijective correspondence between maximal ideals of the underlying C*-algebra and maximal ideals of the reduced crossed product.

Finally, Kawabe \cite{K2017} extended the methods introduced in this paper to undertake a systematic study of the ideal structure of reduced crossed products. In particular, he obtains necessary and sufficient conditions for a commutative C*-algebra equipped with an action of a discrete group to separate ideals in the corresponding reduced crossed product. We also mention a recent paper of Bryder \cite{B2017} that applies similar ideas to investigate the structure of reduced crossed products of noncommutative C*-algebras.

\subsection*{Acknowledgements}
The author is grateful to Adrien Le Boudec, Emmanuel Breuillard, Pierre-Emmanuel Caprace, Kenneth Davidson, Eli Glasner, Mehrdad Kalantar, Narutaka Ozawa and Sven Raum for many helpful comments and suggestions.

%%%%%%%%%%%%%%%%%%%%%%%%%%%%%%%%%%%%%%%%%%
\section{Preliminaries} \label{sec:preliminaries}

\subsection{The reduced C*-algebra}

Let $G$ be a discrete group with identity element $e$. Let $\lambda$ denote the left regular representation of $G$ on the Hilbert space $\ell^2(G)$. The {\em reduced C*-algebra} $\ca_\lambda(G)$ is the norm closed algebra generated by the image of $G$ under $\lambda$.

Let $\{\delta_g \mid g \in G\}$ denote the standard orthonormal basis for $\ell^2(G)$. Then every element $a \in \ca_\lambda(G)$ has a Fourier expansion
\[
a = \sum_{g \in G} \alpha_g \lambda_g,
\]
uniquely determined by $\alpha_g = \langle a \delta_e, \delta_g \rangle$ for $g \in G$.

A linear functional $\phi$ on $\ca_\lambda(G)$ is said to be a {\em state} if it is unital and positive, i.e. $\phi(1) = 1$ and $\phi(a) \geq 0$ for every $a \in \ca_\lambda(G)$ with $a \geq 0$. If, in addition, $\phi(ab) = \phi(ba)$ for all $a,b \in \ca_\lambda(G)$, then $\phi$ is said to be {\em tracial}. The C*-algebra $\ca_\lambda(G)$ is always equipped with a {\em canonical tracial state} $\tau_\lambda$ defined by $\tau_\lambda(a) = \langle a \delta_e, \delta_e \rangle$.

For general facts about group C*-algebras, crossed products and completely positive maps, we refer the reader to Brown and Ozawa's book \cite{BO2008}.

\subsection{C*-simplicity and the unique trace property}

A discrete group $G$ is said to be {\em C*-simple} if the reduced C*-algebra $\ca_\lambda(G)$ has no closed non-trivial two-sided ideals. If $\tau_\lambda$ is the unique tracial state on $\ca_\lambda(G)$, then $G$ is said to have the {\em unique trace property}.

The C*-simplicity of $G$ is equivalent to the following property: every representation of $G$ that is weakly contained in the left regular representation is actually weakly equivalent to the left regular representation.

In recent work with Kalantar \cite{KK2014}, we established the following dynamical characterization of C*-simplicity. See Section \ref{sec:group-actions} below for the definition of a boundary action.

\begin{thm} \label{thm:c-star-simple-iff-free-boundary-action}
A discrete group is C*-simple if and only if it has a (topologically) free boundary action.
\end{thm}

In more recent work with Breuillard, Kalantar and Ozawa \cite{BKKO2014}, we applied Theorem \ref{thm:c-star-simple-iff-free-boundary-action} to prove that many groups are C*-simple. We observed that certain strong group-theoretic conditions are imposed on any group that does not have a topologically free boundary action. However, we also observed that these conditions do not characterize C*-simplicity.

In addition, we established the following characterization of the unique trace property.

\begin{thm}
A discrete group has the unique trace property if and only if it has no non-trivial normal amenable subgroups. In particular, every discrete C*-simple group has the unique trace property.
\end{thm}

It is well known that a C*-simple discrete group necessarily has no normal amenable subgroups (see e.g. \cite{H2007}).

\subsection{Group actions}\label{sec:group-actions}

Let $G$ be a discrete group. A compact Hausdorff space $X$ is said to be a {\em $G$-space} if $G$ acts by homeomorphisms on $X$. This is equivalent to the existence of a group homomorphism from $G$ into the group of homeomorphisms on $X$. For $g \in G$ and $x \in X$ we will write $gx$ for the image of $x$ under $g$.

A C*-algebra $A$ is said to be a {\em $G$-C*-algebra} if $G$ acts by automorphisms on $A$. This is equivalent to the existence of a group homomorphism from $G$ into the group of automorphisms of $A$. For $g \in G$ and $a \in A$ we will write $ga$ for the image of $a$ under $g$.

An example of a $G$-C*-algebra is provided by the C*-algebra $C(X)$ of continuous functions on a $G$-space $X$. The $G$-action on $C(X)$ is given by
\[
(gf)(x) = f(g^{-1}x), \quad g \in G,\ f \in C(X),\ x \in X.
\]

Another example of a $G$-C*-algebra is provided by the reduced C*-algebra $\ca_\lambda(G)$. The $G$-action on $\ca_\lambda(G)$ is given by
\[
ga = \lambda_g a \lambda_g^{-1}, \quad g \in G,\ a \in A.
\]

If $A$ is a $G$-C*-algebra, then the dual space $A^*$, equipped with the weak* topology, is a $G$-space with respect to the action
\[
(g\phi)(a) = \phi(g^{-1}a), \quad g \in G,\ a \in A,\ \phi \in A^*.
\]
If $A$ is unital, then the state space $S(A)$ of $A$ is a weak*-closed $G$-invariant subset of $A^*$. In particular, if $X$ is a $G$-space, then the space $P(X)$ of probability measures on $X$ is a $G$-space since it can be identified with the state space $S(C(X))$ of $C(X)$.

A $G$-space $X$ is said to be a {\em $G$-boundary} if for every $\mu \in P(X)$ and $x \in X$, the point mass $\delta_x$ belongs to the weak* closure of the orbit $G\mu$. There is always a unique boundary $\fb$ called the {\em Furstenberg boundary} of $G$ \cite{F1973}*{Section 4} that is universal in the sense that for every $G$-boundary $X$, there is a surjective $G$-equivariant map from $\fb$ onto $X$.

A convex $G$-space $K$ is said to be {\em affine} if
\[
g(\alpha x + (1-\alpha) y) = \alpha gx + (1-\alpha)gy,
\]
for all $g \in G$, $0 \leq \alpha \leq 1$ and $x,y \in K$. If, in addition, $K$ contains no proper affine $G$-space, then $K$ is said to be {\em minimal}. By Zorn's lemma, every affine $G$-space contains a minimal affine $G$-space. We will frequently use the fact from \cite{G1976}*{Theorem III.2.3} that if $K$ is a minimal affine $G$-space, then the closure $\cl \ex K$ of the set of extreme points of $K$ is a $G$-boundary.

%%%%%%%%%%%%%%%%%%%%%%%%%%%%%%%%%%%%%%%%%%
\section{Boundary maps and boundary states} \label{sec:boundary-maps-boundary-states}

In \cite{BKKO2014}, it was shown that a group has the unique trace property, meaning that its reduced C*-algebra has a unique tracial state, if and only if it has no non-trivial normal amenable subgroups. In particular, every C*-simple group has the unique trace property. While the results of Le Boudec \cite{L2015} imply that the converse does not hold, questions remain about the exact relationship between C*-simplicity and the unique trace property. In this section, we will completely resolve these questions.

Let $G$ be a discrete group with Furstenberg boundary $\fb$. By identifying the scalars with the the constant functions in $C(\fb)$, every tracial state on the reduced C*-algebra $\ca_\lambda(G)$ can be viewed as a $G$-equivariant unital completely positive map from $\ca_\lambda(G)$ to $C(\fb)$. The existence of the canonical trace $\tau_\lambda$ on $\ca_\lambda(G)$ ensures that there is always at least one such map.

The next result gives a complete description of the $G$-equivariant unital completely positive maps from $\ca_\lambda(G)$ to $C(\fb)$ in terms of $G$-boundaries in the state space $S(\ca_\lambda(G))$ of the reduced C*-algebra $\ca_\lambda(G)$, i.e. weak* compact $G$-invariant subsets of $S(\ca_\lambda(G))$ that are $G$-boundaries.

Since the canonical tracial state $\tau_\lambda$ is $G$-invariant, the singleton $\{\tau_\lambda\}$ is trivially a $G$-boundary in $S(\ca_\lambda(G))$. We will say that a $G$-boundary $X \subset S(\ca_\lambda(G))$ is {\em non-trivial} if $X \ne \{\tau_\lambda\}$. Note that if $\phi$ is a non-canonical tracial state on $\ca_\lambda(G)$, then the singleton $\{\phi\}$ is $G$-equivariant, and hence is a non-trivial $G$-boundary.

\begin{prop} \label{prop:correspondence-maps-boundaries}
For a discrete group $G$, there is a bijective correspondence between $G$-equivariant unital completely positive maps from $\ca_\lambda(G)$ to $C(\fb)$ and $G$-boundaries in the state space of $\ca_\lambda(G)$.
\end{prop}

\begin{proof}
Let $\Phi : \ca_\lambda(G) \to C(\fb)$ be a $G$-equivariant unital completely positive map and let $\Phi^* : P(\fb) \to S(\ca_\lambda(G))$ denote the restriction of the adjoint of $\phi$ to the space $P(\fb)$ of probability measures on $\fb$. By \cite{G1976}*{Proposition III.2.4}, the range $K = \Phi^*(P(\fb))$ of $\Phi^*$ is a minimal affine $G$-space, and by \cite{G1976}*{Theorem III.2.3}, the closure $X = \cl \ex K$ of the set of extreme points of $K$ is a $G$-boundary.

Conversely, if $X \subset S(\ca_\lambda(G))$ is a $G$-boundary, then by the universality of $\fb$, there is a surjective $G$-equivariant map $\phi : \fb \to X$. By contravariance, $\phi$ induces a $G$-equivariant unital completely positive map $\Phi : \ca_\lambda(G) \to C(\fb)$ defined for $a \in \ca_\lambda(G)$ by
\[
\Phi(a)(x) = \phi(x)(a), \quad x \in \fb. \qedhere
\]
\end{proof}

\begin{prop} \label{prop:canon-non-canon-expectation}
For a discrete group $G$ there is a $G$-equivariant unital completely positive map $\phi : \ca_\lambda(G) \to C(\fb)$ satisfying $\phi(\lambda_s) = \chi_{\operatorname{Fix}(s)}$ for $s \in G$, where $\operatorname{Fix}(s) = \{x \in \fb : sx = x\}$ denotes the set of points in $\fb$ that are fixed by $s$ for $s \in G$.
\end{prop}

\begin{proof}
For $x \in \fb$, the stabilizer subgroup $G_x$ is amenable by \cite{BKKO2014}*{Proposition 2.7}). Hence the indicator function $\chi_{G_x}$ extends by linearity to a state on $\ca_\lambda(G)$ that we continue to denote by $\chi_{G_x}$ (see e.g. \cite{BO2008}*{Chapter 2}).

The set $X = \{\chi_{G_x} : x \in \fb\} \subset S(\ca_\lambda(G))$ is the image of $\fb$ under the map taking $x$ to $\chi_{G_x}$. This map is clearly $G$-equivariant. We claim that it is also continuous. To see this, let $(x_i)$ be a net in $\fb$ with $\lim x_i = x$ for $x \in \fb$. We must show that $\lim \chi_{G_{x_i}} = \chi_{G_x}$.

By \cite{KK2014}*{Remark 3.16}, $\fb$ is extremally disconnected, so Frol{\'\i}k's theorem \cite{Fro1971} implies that for $s \in G$, $\operatorname{Fix}(s)$ is clopen. Hence $\lim \chi_{G_{x_i}}(\lambda_s) = \chi_{G_x}(\lambda_s)$. By linearity, $\lim \chi_{G_{x_i}}(a) = \chi_{G_x}(a)$ for every $a \in \operatorname{span}\{\lambda_s : s \in G\}$. Since this subset is dense in $\ca_\lambda(G)$ and the net $(\chi_{G_{x_i}})$ is bounded, it follows that $\lim \chi_{G_{x_i}} = \chi_{G_x}$.

Since $X$ is the image of $\fb$ under a continuous $G$-equivariant map, it is a boundary in the state space of $\ca_\lambda(G)$. It is easy to check that the map $\phi : \ca_\lambda(G) \to C(\fb)$ is the corresponding $G$-equivariant unital completely positive map constructed as in Proposition \ref{prop:correspondence-maps-boundaries}.
\end{proof}

\begin{prop} \label{prop:unique-map-iff-free}
For a discrete group $G$ the only $G$-equivariant unital completely positive map from $\ca_\lambda(G)$ to $C(\fb)$ is the canonical trace $\tau_\lambda$ if and only if the action on $\fb$ is free.
\end{prop}

\begin{proof}
The forward direction follows from the fact that the map constructed in Proposition \ref{prop:canon-non-canon-expectation} agrees with the canonical trace on $\ca_\lambda(G)$ if and only if the action on $\fb$ is free.

For the other direction, let $\phi : \ca_\lambda(G) \to C(\fb)$ be a $G$-equivariant unital completely positive map. Suppose the action on $\fb$ is free. Proceeding as in the proof of \cite{BKKO2014}*{Theorem 3.1}, we can extend $\phi$ to a $G$-equivariant unital completely positive map $\psi : C(\fb) \times_r G \to C(\fb)$ such that $C(\fb)$ belongs to the multiplicative domain of $\psi$.

By \cite{KK2014}*{Remark 3.16}, $\fb$ is extremally disconnected, so Frol{\'\i}k's theorem \cite{Fro1971} implies that for $s \in G \setminus \{e\}$ and $x \in \fb$ there is a clopen subset $U \subset \fb$ such that $x \in U$ and $x\notin sU$. Let $\chi_U \in C(\fb)$ denote the indicator function for $U$. Since $C(\fb)$ belongs to the multiplicative domain of $\psi$,
\begin{multline*}
\phi(\lambda_s)(x) = \phi(\lambda_s)(x)\chi_U(x) = \psi(\lambda_s \chi_U)(x) = \psi(\chi_{sU} \lambda_s)(x) \\
= \chi_{sU}(x)\phi(\lambda_s)(x) = 0.
\end{multline*}
Hence $\phi(\lambda_s) = 0$ and we conclude that $\phi$ agrees with the canonical trace on $\ca_\lambda(G)$.
\end{proof}

By \cite{KK2014}*{Theorem 6.2}, the freeness of the action on $\fb$ is equivalent to the C*-simplicity of $G$. Hence Proposition \ref{prop:unique-map-iff-free} implies the following result, which roughly says that a discrete group is C*-simple if and only if the only ``trace-like'' map on its reduced C*-algebra is the canonical tracial state.

\begin{thm} \label{thm:c-star-simple-iff-unique-map}
A discrete group $G$ is C*-simple if and only if the only $G$-equivariant unital completely positive map from $\ca_\lambda(G)$ to $C(\fb)$ is the canonical trace $\tau_\lambda$.
\end{thm}

\begin{rem}
If $G$ is a non-C*-simple discrete group with the unique trace property, then Theorem \ref{thm:c-star-simple-iff-unique-map} implies the existence of a $G$-equivariant unital completely positive map from $\ca_\lambda(G)$ to $C(\fb)$ that is not scalar-valued.
\end{rem}

Combining Proposition \ref{prop:correspondence-maps-boundaries} and Theorem \ref{thm:c-star-simple-iff-unique-map} gives the following result.

\begin{thm} \label{thm:c-star-simple-iff-no-non-trivial-G-boundaries}
A discrete group $G$ is C*-simple if and only if there are no non-trivial $G$-boundaries in the state space of $\ca_\lambda(G)$.
\end{thm}

Observe that Theorem \ref{thm:c-star-simple-iff-no-non-trivial-G-boundaries} clarifies the relationship between C*-simplicity and the unique trace property. In particular, since non-canonical traces give rise to non-trivial boundaries in the state space of the reduced C*-algebra, Theorem \ref{thm:c-star-simple-iff-no-non-trivial-G-boundaries} implies the result from \cite{BKKO2014}*{Corollary 4.3} that discrete C*-simple groups have the unique trace property.

\begin{cor} \label{cor:closed-convex-hull-contains-trace}
A discrete group $G$ is C*-simple if and only if for every state $\phi$ on $\ca_\lambda(G)$, the weak* closed convex hull of $G\phi$ contains the canonical trace $\tau_\lambda$ on $\ca_\lambda(G)$.
\end{cor}

\begin{proof}
By Theorem \ref{thm:c-star-simple-iff-no-non-trivial-G-boundaries}, $G$ is C*-simple if and only if the only $G$-boundary in the state space of $\ca_\lambda(G)$ is the singleton $\{\tau_\lambda\}$. The result follows from the fact that for every state $\phi$ on $\ca_\lambda(G)$, the weak* closed convex hull of $G\phi$ contains a minimal affine $G$-space $K$, and by \cite{G1976}*{Theorem III.2.3}, the closure $\cl \ex K$ of the set of extreme points of $K$ is a $G$-boundary. 
\end{proof}

More generally, we need to consider boundaries in the dual $\ca_\lambda(G)^*$ of $\ca_\lambda(G)$, i.e. weak* compact $G$-invariant subsets that are $G$-boundaries. We will say that a $G$-boundary $X \subset \ca_\lambda(G)^*$ is {\em non-trivial} if $X \ne \{\phi\}$, where $\phi$ is a scalar multiple of the canonical tracial state $\tau_\lambda$.

The next result strengthens Theorem \ref{thm:c-star-simple-iff-no-non-trivial-G-boundaries}.

\begin{thm} \label{thm:c-star-simple-iff-no-non-trivial-G-boundaries-in-dual}
A discrete group $G$ is C*-simple if and only if there are no non-trivial $G$-boundaries in the dual space of $\ca_\lambda(G)$.
\end{thm}

\begin{proof}
One direction is clearly implied by Theorem \ref{thm:c-star-simple-iff-no-non-trivial-G-boundaries}.

For the other direction, suppose $G$ is a C*-simple discrete group. Let $K$ be a minimal affine $G$-space in $\ca_\lambda(G)^*$ and fix $\phi \in K$. We must show that $\phi$ is a scalar multiple of the canonical trace $\tau_\lambda$.

By the Hahn-Jordan decomposition for bounded linear functionals in C*-algebras, we can write $\phi = \phi_1 - \phi_2 + i(\phi_3 - \phi_4)$, where each $\phi_i$ is a non-negative scalar multiple of a state on $\ca_\lambda(G)$.

By Corollary \ref{cor:closed-convex-hull-contains-trace}, the weak* closed convex hull of $G\phi_1$ contains $\phi_1(1) \tau_\lambda$. Hence there is a net $(\alpha_j)$ of non-negative finitely supported sequences in $\bR^G$ such that $\sum_{s \in G} \alpha_j(s) = 1$ for each $j$ and
\[
\lim_j \sum_{s \in G} \alpha_j(s) s \phi_1 = \phi_1(1) \tau_\lambda. 
\]

By applying compactness and passing to a subnet, we can assume that for each $i=2,3,4$, $\sum_{s \in G} \alpha_j(s) s \phi_i$ also converges to a non-negative multiple of a state, say $\psi_i$. Hence $\phi_1(1) \tau_\lambda - \psi_2 + i(\psi_3 - \psi_4) \in K$. Since $\tau_\lambda$ is $G$-invariant, it follows by the same argument that $\phi_1(1)\tau_\lambda - \psi_2(1)\tau_\lambda + i(\psi_3' - \psi_4') \in K$, where $\psi_3'$ and $\psi_4'$ are non-negative multiplies of states. Applying this argument two more times, we conclude that $K$ contains a scalar multiple of $\tau_\lambda$. Since $K$ is minimal and $\tau_\lambda$ is $G$-invariant, the result now follows.
\end{proof}

Applying Theorem \ref{thm:c-star-simple-iff-no-non-trivial-G-boundaries-in-dual} and arguing as in the proof of Corollary \ref{cor:closed-convex-hull-contains-trace}  gives the following result. 

\begin{cor} \label{cor:closed-convex-hull-contains-trace-dual}
A discrete group $G$ is C*-simple if and only if for every bounded linear functional $\phi$ on $\ca_\lambda(G)$,  the weak* closed convex hull of $G\phi$ contains $\phi(1)\tau_\lambda$, where $\tau_\lambda$ denotes the canonical trace on $\ca_\lambda(G)$.
\end{cor}

%%%%%%%%%%%%%%%%%%%%%%%%%%%%%%%%%%%%%%%%%%
\section{Uniformly recurrent subgroups} \label{sec:uniformly-recurrent-subgroups}

Let $G$ be a discrete group and let $\S(G)$ denote the compact space of subgroups of $G$ equipped with the {\em Chabauty topology}, which coincides with the product topology on $\{0,1\}^G$.

Convergence in the Chabauty topology can be described in the following way: a net of subgroups $(H_i) < G$ converges in the Chabauty topology to a subgroup $H < G$ if
\begin{enumerate}
\item every $h \in H$ eventually belongs to $H_i$ and
\item for every subnet $(H_j)$, $\cap_j H_j \subset H$.
\end{enumerate}

The space $\S(G)$ is a $G$-space with respect to the conjugation action of $G$. Let $\S_a(G)$ denote the $G$-subspace of amenable subgroups of $G$. It is clear that $\S_a(G)$ is $G$-invariant. By \cite{S1971}, $\S_a(G)$ is closed.

The notion of a uniformly recurrent subgroup of $G$ was recently introduced by Glasner and Weiss \cite{GW2015} as a topological-dynamical analogue of the notion of an invariant random subgroup. A $G$-subspace $X \subset \S(G)$ is said to be a {\em uniformly recurrent subgroup} of $G$ if it is minimal, i.e. if $\{gHg^{-1} \mid g \in G\}$ is dense in $X$ for every $H \in X$. If $X \subset \S_a(G)$, then $X$ is said to be {\em amenable}. If $X \ne \{\{e\}\}$, where $\{e\}$ denotes the trivial subgroup of $G$, then $X$ is said to be {\em non-trivial}.

\begin{thm} \label{thm:c-star-simple-iff-no-amenable-urs}
A discrete group $G$ is C*-simple if and only if it has no non-trivial amenable uniformly recurrent subgroups.
\end{thm}

\begin{proof}
($\Leftarrow$) Suppose that $G$ is not C*-simple. Let $X = \{G_x \mid x \in \fb\}$, where $G_x$ denotes the stabilizer subgroup of a point $x \in \fb$. Then $X$ is the image of $\fb$ under the map taking $x$ to $G_x$. This map is $G$-equivariant, and it is not difficult to check that it is continuous using the fact from \cite{KK2014}*{Remark 3.16} that $\fb$ is extremally disconnected.

Since $\fb$ is compact and minimal, $X$ is also compact and minimal. Furthermore, by \cite{BKKO2014}*{Lemma 2.3}, $G_x$ is amenable for each $x \in \fb$. Hence by \cite{S1971}, $X \subset \S_a(G)$, and thus $X$ is an amenable uniformly recurrent subgroup.

By \cite{KK2014}*{Theorem 6.2}, since $G$ is not C*-simple, the $G$-action on $\fb$ is not topologically free. In particular, $G_x \ne \{e\}$ for some $x \in \fb$, and hence $X$ is non-trivial.

($\Rightarrow$)
Suppose that $G$ has a non-trivial amenable uniformly recurrent subgroup $X$. Fix $H \in X$. Then $H$ is amenable and by minimality, $X = \cl \{gHg^{-1} \mid g \in G\}$. Since $H$ is amenable, the indicator function $\chi_H$ extends to a state $\phi$ on $\ca_\lambda(G)$.

For finite $F \subset G \setminus \{e\}$, let $U_F = \{H \in S(G) : H \cap F = \emptyset \}$. Then the family $\{U_F : F \subseteq G \setminus \{e\} \text{ finite}\}$ is a neighborhood basis for the trivial subgroup $\{e\}$ in $S(G)$. Since $X$ is non-trivial, $\{e\} \notin X$. Hence there is finite $F \subset G \setminus \{e\}$ such that $F \cap gHg^{-1} \ne \emptyset$ for all $g \in G$.

Let $a = \sum_{s \in F} \lambda_s \in \ca_\lambda(G)$. Then for $g \in G$,
\[
(g\phi)(a) = \sum_{s \in F} \chi_{gHg^{-1}}(s) = |F \cap gHg^{-1}| \geq 1.
\]
Letting $K$ be a minimal affine $G$-subspace of the closed convex hull of $G \phi$, it follows that for every $\psi \in K$, $\psi(a) \geq 1$. By \cite{G1976}*{Theorem III.2.3}, the closure $Y = \cl \ex K$ of the set of extreme points of $K$ is a $G$-boundary, and $Y$ does not contain the canonical trace $\tau_\lambda$. In particular, the state space of $\ca_\lambda(G)$ contains a non-trivial $G$-boundary, and the result now follows by Theorem \ref{thm:c-star-simple-iff-no-non-trivial-G-boundaries}.
\end{proof}

\begin{rem}
Pierre-Emmanuel Caprace and Adrien Le Boudec kindly suggested the following alternative proof of the forward implication of Theorem \ref{thm:c-star-simple-iff-no-amenable-urs}.

Suppose that $G$ has a non-trivial amenable uniformly recurrent subgroup $X$. Fix $H \in X$. Then $H$ is amenable and by minimality, $X = \cl \{gHg^{-1} \mid g \in G\}$. Consider the space $P(\fb)$ of probability measures on $\fb$. Fix $x \in \fb$ and let $\delta_x \in P(\fb)$ denote the corresponding point mass. By the amenability of $H$, there is a probability measure $\mu \in P(\fb)$ fixed by $H$. Since $\fb$ is a boundary, there is a net $(g_i) \in G$ such that $\lim_i g_i \mu = \delta_x$.

By compactness, after passing to a subnet we can suppose that $g_i H g_i^{-1}$ converges to $K \in \S(G)$, and since $X$ is non-trivial, $K \ne \{e\}$. By the definition of the Chabauty topology, this implies that $\cap_i g_i H g_i^{-1} \ne \{e\}$. Fix $g \in \cap_i g_i H g_i^{-1} \setminus \{e\}$. Then $g g_i \mu = g_i \mu$ for each $i$. Taking the limit gives $g \delta_x = \delta_x$ and hence $gx = x$. Thus $G$ does not act freely on $\fb$, and it follows from \cite{KK2014}*{Theorem 6.2} that $G$ is not C*-simple.
\end{rem}

\begin{rem} \label{rem:c-star-simple-iff-special-urs}
The proof of Theorem \ref{thm:c-star-simple-iff-no-amenable-urs} actually implies the following result: a discrete group $G$ is not C*-simple if and only if the family of point stabilizers $\{G_x \mid x \in \fb\}$ is a non-trivial amenable uniformly recurrent subgroup for $G$. In fact, this family is a $G$-boundary, since it is the image of the Furstenberg boundary under a $G$-equivariant map.
\end{rem}

\begin{example}
For any non-trivial amenable group $G$, the singleton $\{G\}$ is a non-trivial amenable uniformly recurrent subgroup, and hence by Theorem \ref{thm:c-star-simple-iff-no-amenable-urs}, we obtain the well known fact that a non-trivial amenable group is never C*-simple.
\end{example}

\begin{example}
Let $G$ be a discrete group with only countably many amenable subgroups and no non-trivial normal amenable subgroups. It was shown in \cite{BKKO2014}*{Theorem 3.8} that $G$ is C*-simple. We will give another proof using Theorem \ref{thm:c-star-simple-iff-no-amenable-urs}. 

 Let $X$ be an amenable uniformly recurrent subgroup. We must show that $X$ is trivial. By assumption, $X$ is a countable set, and hence it has an isolated point, say $H$. By compactness and minimality, finitely many conjugates of the open set $\{H\}$ cover $X$. Hence $X$ consists of finitely many conjugates of $H$. Since $X$ is $G$-invariant, it follows that $H$ is almost normal, i.e. $H$ has finitely many conjugates in $G$. Equivalently, the normalizer $N_G(H)$ of $H$ has finite index in $G$.
 
Since $N_G(H)$ has finite index, there is a subgroup $N_0 < N_G(H)$ that is normal as a subgroup of $G$ and has finite index in $G$. For $g \in G$, let $N_g = N_0 \cap gHg^{-1}$. Then each $N_g$ is amenable and normal in $N_0$. Let $N < G$ denote the subgroup generated by all of the $N_g$. Then $N$ is amenable since it is generated by finitely many amenable subgroups that are normal in $N_0$. Moreover, $N$ is normal in $G$. Hence by assumption $N$ is trivial. In particular, $N_e = N_0 \cap H$ is trivial.

Since $N_0$ has finite index in $G$, it follows that $H$ is finite. Therefore, the conjugacy class of every element in $H$ must be finite, i.e. $H$ must be a subgroup of the FC-center of $G$. But since the FC-center is amenable and normal in $G$, it is trivial by assumption. Hence $H$ is trivial, and it follows that $X$ is trivial. By Theorem \ref{thm:c-star-simple-iff-no-amenable-urs}, we conclude that $G$ is C*-simple.

This result applies to Tarski monster groups and torsion-free Tarski monster groups, which were shown to be C*-simple in \cite{KK2014}*{Corollary 6.6} and \cite{BKKO2014}*{Corollary 3.9} respectively. In addition, by \cite{I1994}, this result applies to the family of free Burnside groups $B(m,n)$ for $m \geq 2$ and $n$ odd and sufficiently large, which were recently shown to be C*-simple by Olshanskii and Osin \cite{OO2014}.
\end{example}

\begin{rem} \label{rem:invariant-random-subgroups}
An {\em invariant random subgroup} of a discrete group $G$ is a $G$-invariant probability measure $\mu$ on the space of subgroups $\S(G)$, which can be viewed as the distribution of a random subgroup of $G$. This notion was introduced by Ab\'ert, Glasner and Vir\'ag \cite{AGV2014}. 

If $\mu$ is supported on the $G$-subspace $\S_a(G)$ of amenable subgroups of $G$, then $\mu$ is said to be {\em amenable}. If $\mu \ne \delta_{\{e\}}$, where $\delta_{\{e\}}$ denotes the point mass corresponding to the trivial subgroup $\{e\}$, then $\mu$ is said to be {\em non-trivial}.

A group with the unique trace property, and in particular any C*-simple group (see \cite{BKKO2014}*{Corollary 4.3}), does not have non-trivial amenable invariant random subgroups. This can be seen by invoking \cite{T2012}*{Theorem 5.14}, which implies that a non-trivial amenable invariant random subgroup gives rise to a non-canonical tracial state on the reduced C*-algebra as in the proof of Theorem \ref{thm:c-star-simple-iff-no-amenable-urs}.

Thus while Le Boudec's examples from \cite{L2015} of non-C*-simple groups with the unique trace property do not have non-trivial amenable invariant random subgroups, they do have non-trivial amenable uniformly recurrent subgroups.
\end{rem}

%%%%%%%%%%%%%%%%%%%%%%%%%%%%%%%%%%%%%%%%%%
\section{Residually normal subgroups} \label{sec:residually-normal-subgroups}

In this section, by unraveling the definition of a uniformly recurrent subgroup from Section \ref{sec:uniformly-recurrent-subgroups} and invoking Theorem \ref{thm:c-star-simple-iff-no-amenable-urs}, we will establish a more algebraic characterization of C*-simplicity in terms of the existence of amenable subgroups of the following type.

\begin{defn} \label{defn:residually-normal-subgroup}
Let $G$ be a discrete group. A subgroup $H < G$ is said to be {\em residually normal} if there exists a finite subset $F \subset G \setminus \{e\}$ such that $F \cap gHg^{-1} \ne \emptyset$ for all $g \in G$.
\end{defn}

Note that residually normal subgroups are non-trivial, and that every non-trivial normal subgroup is residually normal. The following characterization of residually normal subgroups follows immediately from the description of the Chabauty topology in Section \ref{sec:uniformly-recurrent-subgroups}, as in the proof of the forward implication of Theorem \ref{thm:c-star-simple-iff-no-amenable-urs}.

\begin{prop} \label{prop:residually-normal-subgroups}
Let $G$ be a discrete group and let $H < G$ be a subgroup. The following are equivalent:
\begin{enumerate}
\item The subgroup $H$ is residually normal.
\item For every net $(g_i)$ in $G$ there is a subnet $(g_j)$ such that
\[
\bigcap_j g_j H g_j^{-1} \ne \{e\}.
\]
\item The trivial subgroup $\{e\}$ does not belong to the closure $\overline{\operatorname{Conj}}(H) \subset S(G)$ of the set $\operatorname{Conj}(H) = \{ gHg^{-1} \mid g \in G\}$ of subgroups conjugate to $H$, where the closure is taken in the Chabauty topology.
\end{enumerate}
\end{prop}

Let $G$ be a discrete group. It follows immediately from the definition that the closure of the orbit of an amenable residually normal subgroup $H < G$ in $S(G)$ contains a non-trivial amenable uniformly recurrent subgroup. On the other hand, if $X \subset S(G)$ is an amenable uniformly recurrent subgroup, then every $H \in X$ is amenable and residually normal. The following characterization of C*-simplicity in terms of amenable residually normal subgroups is therefore an immediate consequence of Theorem~\ref{thm:c-star-simple-iff-no-amenable-urs}.

\begin{thm} \label{thm:c-star-simple-iff-no-amenable-residually-normal-subgroups}
A discrete group is C*-simple if and only if it has no amenable residually normal subgroups.
\end{thm}

\begin{rem}
In this remark we consider the relationship between residually normal subgroups and normalish subgroups as introduced in \cite{BKKO2014}.

For a discrete group $G$, a subgroup $H < G$ is said to be {\em normalish} if for any finite sequence $g_1,\ldots,g_n \in G$, the intersection of the conjugates $\cap_n g_nHg_n^{-1}$ is infinite. In particular, every infinite normal subgroup is normalish.

A key result in \cite{BKKO2014}*{Theorem 2.12} implies that a discrete group with no non-trivial finite normal subgroups and no amenable normalish subgroups is C*-simple. However, C*-simple groups can have amenable normalish subgroups (see Example \ref{ex:baumslag-solitar}). Hence by Theorem \ref{thm:c-star-simple-iff-no-amenable-residually-normal-subgroups}, amenable normalish subgroups are not necessarily residually normal.

On the other hand, infinite residually normal subgroups are not necessarily normalish. For instance, if $G$ is a discrete group and $H < G$ is a subgroup with only finitely many conjugates, then $H$ is always residually normal, but $H$ is normalish if and only if it contains an infinite normal subgroup.

However, the following strengthening of Theorem \ref{thm:c-star-simple-iff-no-amenable-residually-normal-subgroups} does hold: a discrete group is C*-simple if and only if it has no amenable residually normal subgroups that are either finite normal or normalish.

To see this, recall from Remark \ref{rem:c-star-simple-iff-special-urs} that a discrete group $G$ is not C*-simple if and only if $\{G_x \mid x \in \fb\}$ is a non-trivial amenable uniformly recurrent subgroup for $G$. In this case, by minimality it follows from Proposition \ref{prop:residually-normal-subgroups} that each $G_x$ is residually normal, and it follows from \cite{BKKO2014}*{Lemma 2.10} that each $G_x$ is amenable and normalish. In this way, we obtain \cite{BKKO2014}*{Theorem 2.12}.
\end{rem}

\begin{example} \label{ex:baumslag-solitar}
In this example we consider the Baumslag-Solitar group $BS(m,n)$ for $|m| \ne |n|$ and $|m|,|n| \geq 2$, defined by
\[
BS(m,n) = \langle a,t \mid t^{-1} a^m t = a^n \rangle.
\]
The cyclic subgroup $\langle a \rangle$ is clearly amenable, and it was observed in \cite{BKKO2014}*{Section 3.7} that it is  normalish. However, it is easy to check that it is not residually normal. In fact, by \cite{HP2011}, $BS(m,n)$ is C*-simple, and hence by Theorem \ref{thm:c-star-simple-iff-no-amenable-residually-normal-subgroups} it has no amenable residually normal subgroups. 
\end{example}

\begin{example} \label{ex:le-boudec}
In this example we consider Le Boudec's recent example \cite{L2015} of non-C*-simple groups with the unique trace property.

Let $T$ be a locally finite tree with boundary $\partial T$. For a fixed vertex $v$ in $T$, points in $\partial T$ correspond to geodesic rays in $T$ starting at $v$. Let $\aut(T)$ denote the automorphism group of $T$, and let $G < \aut(T)$ be a subgroup.

Suppose $G$ has the property that if $T'$ is one of the two trees obtained by deleting an edge of $T$, then there is a non-trivial element $g \in G$ that fixes every vertex in $T'$. Fix $x \in \partial T$, and let $G_x = \{g \in G \mid gx = x\}$ denote the corresponding stabilizer subgroup in $G$. We will show that $G_x$ is residually normal. 

Let $(g_n) \in G$ be a sequence of elements and let $x_n = g_nx$. Then $g_n G_x g_n^{-1} = G_{x_n}$. By local finiteness, there is an edge $f$ incident to $v$ such that there is a subsequence $(x_{n_k})$ with the property that the geodesic ray corresponding to each $x_{n_k}$ contains $f$. Then by the assumptions on $G$,
\[
\bigcap_k g_{n_k} G_x g_{n_k}^{-1} = \bigcap_k G_{x_{n_k}} \ne \{e\}.
\]
Hence $G_x$ is residually normal.

Le Boudec's examples arise as enlargements of groups acting on their Bass-Serre tree, and satisfy all of the above properties. In addition, they have the property that  stabilizer subgroups corresponding to points in the boundary are amenable. Hence by Theorem  \ref{thm:c-star-simple-iff-no-amenable-residually-normal-subgroups}, these groups are not C*-simple.

Furthermore, Le Boudec shows that these groups act minimally and contain two hyperbolic elements without common fixed points. This is well known to imply that there are no non-trivial amenable normal subgroups. Hence by the characterization in \cite{BKKO2014}*{Theorem 4.1}, these groups have the unique trace property.
\end{example}

\begin{example} \label{ex:thompson-group-1}
In this example, we consider the Thompson groups $F$ and $T$, viewed as subgroups of homeomorphisms acting on the interval $[0,1]$. We recall that $F$ arises as the stabilizer subgroup of $T$ corresponding to the point $0$.

We will show that $F$ is residually normal as a subgroup of $T$. To see this, let $(g_n) \in T$ be a sequence of elements. We must show there is a subsequence $(g_{n_k})$ such that
\[
\bigcap g_{n_k} F g_{n_k}^{-1} \ne \{e\}.
\]
For $x \in [0,1]$, let $T_x < T$ denote the stabilizer subgroup corresponding to $x$. Then $F = T_0$. Let $x_n = g_n(0)$. Then $g_n F g_n^{-1} = T_{x_n}$. Let $(x_{n_k})$ be a convergent subsequence. Then there is a dyadic interval $I$ in $[0,1]$ such that $x_{n_k} \notin I$ for all $k$. Taking $f \in T$ such that $f$ fixes every point in $[0,1] \setminus I$ but does not fix some point in $I$, we see that $f \in T_{x_{n_k}}$ for each $k$. Hence $F$ is residually normal.

By Theorem \ref{thm:c-star-simple-iff-no-amenable-residually-normal-subgroups}, it follows that if $T$ is C*-simple, then $F$ is non-amenable. Thus we obtain another proof of Haagerup and Olesen's result \cite{HO2014} that if $T$ is C*-simple, then $F$ is non-amenable.
\end{example}

%%%%%%%%%%%%%%%%%%%%%%%%%%%%%%%%%%%%%%%%%%
\section{Powers' averaging property} \label{sec:powers-averaging-property}

In this section we will establish the equivalence between C*-simplicity and Powers' averaging property. The first result is an equivalence between C*-simplicity and a kind of ``weak'' averaging property.

\begin{prop} \label{prop:weak-powers-type-condition}
A discrete group $G$ is C*-simple if and only if for every bounded linear functional $\phi$ on the reduced C*-algebra $\ca_{\lambda}(G)$ and every $a \in \ca_\lambda(G)$,
\[
\inf_{b \in K} |\phi(b) - \phi(1)\tau_\lambda(a)| = 0,
\]
where $K$ denotes the norm closed convex hull of $\{\lambda_g a \lambda_{g^{-1}} \mid g \in G\}$.
\end{prop}

\begin{proof}
By Corollary \ref{cor:closed-convex-hull-contains-trace-dual}, $G$ is not C*-simple if and only if there is a bounded linear functional $\phi$ on $\ca_\lambda(G)$ such that $\phi(1) \tau_\lambda$ does not belong to the weak* closed convex hull of $G\phi$. By the Hahn-Banach separation theorem, this is equivalent to the existence of $a \in \ca_\lambda(G)$ such that
\[
\inf_{b \in K} |\phi(b) - \phi(1) \tau_\lambda(a)| > 0,
\]
where $K$ denotes the norm closed convex hull of $\{\lambda_g a \lambda_{g^{-1}} \mid g \in G\}$.
\end{proof}

\begin{defn} \label{defn:powers-averaging-property}
A discrete group $G$ is said to have {\em Powers' averaging property} if for every element $a$ in the reduced C*-algebra $\ca_\lambda(G)$ and $\epsilon > 0$ there are $g_1,\ldots,g_n \in G$ such that
\[
\left\| \frac{1}{n} \sum_{i=1}^n \lambda_{g_i} a \lambda_{g_i}^{-1} - \tau_\lambda(a) 1 \right\| < \epsilon,
\]
where $\tau_\lambda$ denotes the canonical tracial state on $\ca_\lambda(G)$.
\end{defn}

\begin{thm}
A discrete group is C*-simple if and only if it has Powers' averaging property.
\end{thm}

\begin{proof}
($\Leftarrow$)
If $G$ has Powers' averaging property, then the following argument of Powers \cite{P1975} implies that $G$ is C*-simple. Let $I$ be a closed two-sided closed ideal of $\ca_\lambda(G)$. By the faithfulness of $\tau_\lambda$, if $I \ne \{0\}$, then there is $a \in I$ such that $\tau_\lambda(a) = 1$. Applying Powers' averaging property shows that $1 \in I$. Hence $I = \ca_\lambda(G)$ and we see that $\ca_\lambda(G)$ has no non-trivial proper closed two-sided ideals.

($\Rightarrow$)
Suppose that $G$ is C*-simple and fix $a \in \ca_\lambda(G)$. Let $K$ denote the norm closed convex hull of $\{\lambda_g a \lambda_g^{-1} \mid g \in G\}$. We must show that $\tau_\lambda(a) 1 \in K$. Suppose to the contrary that $\tau_\lambda(a) 1 \notin K$. Then by the Hahn-Banach separation theorem there is a bounded linear functional $\phi$ on $\ca_\lambda(G)$ such that
\[
\inf_{b \in K} |\phi(b) - \phi(1) \tau_\lambda(a)| > 0.
\]
Since $G$ is C*-simple, this contradicts Proposition \ref{prop:weak-powers-type-condition}.
\end{proof}

%%%%%%%%%%%%%%%%%%%%%%%%%%%%%%%%%%%%%%%%%%

\end{document}